\documentclass[11pt]{article}
\usepackage[margin=1.0in]{geometry}
\usepackage[utf8]{inputenc}
\usepackage{mathrsfs,amssymb,amsfonts,amsthm,amsmath,amsbsy}
\usepackage{mathtools}
\usepackage{mathabx}
\usepackage{graphicx}
\usepackage{wrapfig}
\usepackage{caption}
\usepackage{mathtools}
\usepackage{subfig}
\usepackage{dsfont}
\usepackage{bbm}
\usepackage{float}
\usepackage{upgreek}
\usepackage{listings}
\usepackage{lipsum}
\usepackage[export]{adjustbox}
\usepackage{titlesec}
\usepackage{hyperref}
\hypersetup{
	colorlinks,
	citecolor=blue,
	linkcolor=blue,
	urlcolor=green}
\newtheorem{theorem}{Theorem}[section]

\newtheorem{corollary}[theorem]{Corollary}
\newtheorem{lemma}[theorem]{Lemma}

\newtheorem{definition1}[theorem]{Definition}

\newtheorem{remark1}[theorem]{Remark}
\newenvironment{remark}{\begin{remark1}\rm}{\hfill$\square$\end{remark1}}

\usepackage[ruled,vlined]{algorithm2e}
\usepackage[font=small]{caption}

\newcommand{\x}{\mathbf{x}}
\newcommand{\y}{\mathbf{y}}
\newcommand{\z}{\mathbf{z}}

\newcommand{\vv}{\mathbf{v}}
\newcommand{\w}{\mathbf{w}}
\newcommand{\0}{\mathbf{0}}
\usepackage[ruled,vlined]{algorithm2e}
\usepackage[font=small]{caption}
\usepackage[usenames,dvipsnames]{xcolor}
\usepackage{tikz}
\usepackage{subfig}
\usepackage{dsfont,fancyhdr}
\usetikzlibrary{matrix}
\usepackage{stmaryrd}
\usepackage{color}

\allowdisplaybreaks

\title{Triple collisions on a comb graph}

\author{David A. Croydon\footnote{Research Institute for Mathematical Sciences, Kyoto University, croydon@kurims.kyoto-u.ac.jp} and Umberto De Ambroggio\footnote{Department of Mathematics, National University of Singapore, umberto@nus.edu.sg}}

\begin{document}

\maketitle

\begin{abstract}
In this article, we consider the number of collisions of \textit{three} independent simple random walks on a subgraph of the two-dimensional square lattice obtained by removing all horizontal edges with vertical coordinate not equal to 0 and then, for $n\in \mathbb{Z}$, restricting the vertical segment of the graph located at horizontal coordinate $n$ to the interval $\{0,1,\dots,\log^{\alpha}(|n|\vee 1)\}$. Specifically, we show the following phase transition: when $\alpha\leq 1$, the three random walks collide infinitely many times almost-surely, whereas when $\alpha>1$, they collide only finitely many times almost-surely. This is a variation of a result of Barlow, Peres and Sousi, who showed a similar phase transition for \textit{two} random walks when the vertical segments are truncated at height $|n|^{\alpha}$.\\
\textbf{MSC:} 60J10 (primary), 05C81, 60J35.\\
\textbf{Keywords:} simple random walk, comb graph, collisions, heat kernel estimates.
\end{abstract}

\section{Introduction}

A given number $k\in \mathbb{N}\coloneqq \{1,2,\dots,\}$ of discrete-time random processes $X^1=(X^1_n)_{n\geq 0},\dots, X^k=(X^k_n)_{n\geq 0}$ on a graph $\mathbb{G}=(V,E)$ are said to \textit{collide} at time $n\in \mathbb{N}_0\coloneqq \mathbb{N}\cup \{0\}$ if $X^1,\dots,X^k$ are at the same vertex at time $n$, i.e. 
\[X^1_n=\dots=X^k_n\in V.\]
The problem of understanding \textit{random walk} collisions in particular, where each $X^i$ is a simple random walk on a connected, locally finite graph $\mathbb{G}$, has attracted considerable interest in recent years. This follows the surprising result of Krishnapur and Peres \cite{KP}, which provided an example of a \textit{recurrent} graph where two independent simple random walks collide only \textit{finitely} many times, almost-surely; this is in contrast with the almost-sure infinitely many collisions property of two independent simple random walks on recurrent transitive graphs. More specifically, Krishnapur and Peres \cite{KP} considered the graph $\text{Comb}(\mathbb{G},\mathbb{Z})$, constructed by attaching a copy of $\mathbb{Z}$ to each vertex of a recurrent infinite graph $\mathbb{G}$ with constant vertex degree and showed that, within this deterministic geometry, two independent simple random walks meet finitely many times with probability one. For example, the latter result holds for $\text{Comb}(\mathbb{Z},\mathbb{Z})$, which informally speaking is just $\mathbb{Z}^2$ with all horizontal edges removed with the exception of those whose vertical coordinate is equal to 0.

Subsequently, Barlow, Peres and Sousi \cite{BPS} gave a simple criterion in terms of the Green's function for two random walks to have the infinite collision property. As a particular example, they investigated the number of collisions on the subgraph $\text{Comb}_{p}(\mathbb{Z},\alpha)$ of $\text{Comb}(\mathbb{Z},\mathbb{Z})$ obtained by retaining those vertices $(n,x)$ with $n\in \mathbb{Z}$, $0 \leq x\leq |n|^{\alpha}$ (the letter `$p$' stands for `polynomial', referring to the height at which we truncate the vertical segments, also called the teeth). They showed that two independent simple random walks on $\text{Comb}_{p}(\mathbb{Z},\alpha)$ collide infinitely many times with probability one if $\alpha\leq 1$, whereas they collide only finitely often almost-surely when $\alpha>1$. Intuitively, the shorter the teeth ($\alpha\leq 1)$, the easier it is for the two random walks to collide (or alternatively, longer teeth make it harder for the walkers to meet). This result was refined by Chen and Chen in \cite{Chenchen}. Specifically, the latter pair showed that, on the subgraph $\text{Comb}(\mathbb{Z},f)$ of $\text{Comb}(\mathbb{Z},\mathbb{Z})$ with vertex set $\{(n,x):n\in \mathbb{Z}, -f(n)\leq x\leq f(n)\}$, where
\[\sum_{n\in \mathbb{N}}\frac{1}{1\vee \max_{-n\leq i\leq n}f(i)}=\infty,\]
two independent simple random walks collide infinitely many times with probability one.

As noted in both \cite{BPS} and \cite{Chenchen}, it is natural to consider what the corresponding result might be for \textit{three} independent simple random walks. As a step in this direction, Chen and Chen demonstrated in \cite{Chenchen} that if $(f(n))_{n\in \mathbb{Z}}$ is a typical realization of a sequence of independent and identically distributed random variables with law $\mu$ supported on $\mathbb{N}_0$ and with finite mean, then almost-surely three independent simple random walks collide infinitely many times with probability one on $\text{Comb}(\mathbb{Z},f)$.

In this work, in the spirit of the result of Barlow, Peres and Sousi \cite{BPS}, we seek to identify the growth rate of a wedge comb that separates the regime in which three random walks collide infinitely often, and the regime in which they do not. (A priori, since we do not have a simple monotonicity property for the infinite collision property, cf. \cite[Question 7]{BPS}, it is not clear that a simple phase transition exists.) To state our result, we introduce the subgraph $\text{Comb}_{\ell}(\mathbb{Z},\alpha)$ (the `$\ell$' here stands for `logarithmic') of $\text{Comb}(\mathbb{Z},\mathbb{Z})$ obtained by retaining those vertices $(n,x)$ with $n\in \mathbb{Z}$, $0\leq x\leq 0\vee\log^{\alpha}(|n|)$. Let $X$, $Y$ and $Z$ be independent simple random walks on $\text{Comb}_{\ell}(\mathbb{Z},\alpha)$, started at the origin vertex $\0\coloneqq (0,0)$. Our goal is to show that: 
\begin{itemize}
    \item if $\alpha> 1$, then the three random walks meet only finitely many times with probability one;
    \item if $\alpha\leq 1$, then the three random walks meet infinitely often with probability one
\end{itemize}
Specifically, our main result is the following.

\begin{theorem}\label{mainthm}
    Let $(X_n)_{n\geq0}, (Y_n)_{n\geq0}$ and $(Z_n)_{n\geq0}$ be independent simple random walks on the graph $\mathrm{Comb}_{\ell}(\mathbb{Z},\alpha)$ defined above. Denoting by $\mathbb{P}^{\mathbf{x},\mathbf{y},\mathbf{z}}(\cdot)$ the law of $(X,Y,Z)$ given that $X_0=\mathbf{x}$, $Y_0=\mathbf{y}$, $Z_0=\mathbf{z}$ for $\mathbf{x},\mathbf{y},\mathbf{z}\in \mathrm{Comb}_{\ell}(\mathbb{Z},\alpha)$, we have:
    \begin{enumerate}
        \item [(a)] $\mathbb{P}^{\0,\0,\0}(|\{n\in \mathbb{N}_0:X_n=Y_n=Z_n\}|<\infty )=1$ if $\alpha> 1$;
        \item [(b)] $\mathbb{P}^{\0,\0,\0}(|\{n\in \mathbb{N}_0:X_n=Y_n=Z_n\}|=\infty)=1$ if $\alpha\leq 1$.
    \end{enumerate}
\end{theorem}

Establishing part (a) is quite straightforward. In particular, relatively standard heat kernel (transition density) estimates yield that, when $\alpha>1$, the expected number of collisions between three independent random walks on $\text{Comb}_{\ell}(\mathbb{Z},\alpha)$ is finite. Clearly, the result follows from this. We note that it was commented in \cite[Question 1]{BPS} that a similar argument applies to the graph $\text{Comb}_{p}(\mathbb{Z},\alpha)$ for any $\alpha>0$. As we observe below (see Remark \ref{4collisionrem}), the same argument can be used to show that four (or more) random walks collide finitely often on any infinite, connected graph of bounded degree.

To establish (b), we will use a second moment method argument to bound from below the probability of seeing a triple collision prior to the first time  $\theta_N$ at which one of the walkers leaves a given `box' (more precisely, a specific subgraph) in the comb with diameter of order $N$. In particular, for $\alpha\leq 1$, we show that the three random walkers, started at arbitrary sites in the box, collide before time  $\theta_N$ with probability of order at least $1/ \log(N)\log\log(N)$. Whilst this is already enough to establish the infinite collision property, subsequently we show (again by means of a second moment argument) that, if the three walkers meet before time $\theta_N$, they go on to meet at least $c(\log^{1-\alpha}(N)\vee \log\log(N))$ times before $\theta_N$, with constant positive probability. Combining these two estimates we arrive at the conclusion that the three random walkers meet at least $c(\log^{1-\alpha}(N)\vee \log\log(N))$ times before $\theta_N$ with probability of order at least $1/\log(N)\log\log(N)$. By means of the strong Markov property and a generalised Borel-Cantelli argument, we use this lower bound to establish that the two walkers collide with the aforementioned frequency in infinitely many of these (growing) boxes; more precisely, our argument yields that if $C_N$ is the number of collisions up to time $N$, then the following result holds. Clearly, it implies that infinitely many collisions between the walkers occur. 

\begin{theorem}\label{thmgrowth} If $\alpha\leq 1$, then it $\mathbb{P}^{\0,\0,\0}$-almost-surely holds that
\[\limsup_{n\rightarrow\infty}\frac{C_N}{\log^{1-\alpha}(N)\vee \log\log(N)}>0.\]
\end{theorem}

\begin{remark}
In the course of the proof of the above theorem, we establish that $\log^{1-\alpha}(N)\vee \log\log(N)$ also gives an upper bound on the expected number of triple collisions at vertices with horizontal coordinate in $[N/2,2N]$ up to time $\theta_N \wedge cN^2\log^\alpha(N)$. (We will additionally explain that $\theta_N$ is typically of order $N^2\log^\alpha(N)$.) One is thus led to consider what more detailed statements about the collision growth rate are available. What is the correct scaling to see a distributional limit for $(C_N)_{N\geq 1}$, if such exists? On what scale are there almost-sure fluctuations? We leave these as open questions.
\end{remark}

\paragraph{Related work.} 
We summarise here some other recent work on random walk collisions. This is only meant to give a flavour of the activity in this area, and will not be fully comprehensive; please see the references within the cited works for further developments. The result of Barlow, Peres and Sousi \cite{BPS} that we mentioned earlier concerning the phase transition in the number of collisions of two independent random walks on $\text{Comb}_{p}(\mathbb{Z},\mathbb{Z})$ is just one of the several results established in that paper. In particular, the authors of \cite{BPS} considered the double collision property on other discrete (possibly random) geometries. For instance, they established, by means of the general criterion mentioned earlier, that each one of the following graphs has the double infinite collision property: (i) critical Galton Watson trees with finite variance conditioned to survive forever; (ii) the incipient infinite cluster of critical percolation in dimension $d \geq 19$; (iii) the uniform spanning tree in $\mathbb{Z}^2$. They also established a phase transition for the number of double collisions on a class of spherically symmetric trees. Another general result, showing the infinite double collision property for recurrent reversible random rooted graphs was established by Hutchcroft and Peres in \cite{HP}; this covers the uniform infinite planar triangulation and quadrangulation, and the incipient infinite cluster of critical percolation on $\mathbb{Z}^2$. Moreover, Halberstam and Hutchcroft \cite{HH} studied the collision problem in a dynamic random conductance model in $\mathbb{Z}^2$, and established conditions under which two random walks collide infinitely often almost-surely (their result applies to random walks on dynamical percolation). Concerning more details about collisions of random walks on $\mathbb{Z}$, Nguyen \cite{Nguyen} provides a scaling limit for the collision measure. Finally, we note that, in contrast to the case of simple random walks on unweighted graphs, it is shown in \cite{DGP1,DGP2} that arbitrary many walkers can meet infinitely often in certain random media. In particular, in the latter work, the authors considered $d$ independent walkers in the same random environment on $\mathbb{Z}$, and, under the assumption that the law of the environment is such that a single walker is transient to the right but sub-ballistic, they established that there are almost-surely infinitely many times for which all the $d$ random walkers are at the same location. 

\paragraph{Notation.} We denote $\mathbb{N}_0\coloneqq \mathbb{N}\cup \{0\}$, where $\mathbb{N}\coloneqq \{1,2,\dots\}$. Given vertices $\mathbf{x},\mathbf{y},\mathbf{z}$ in $\text{Comb}_{\ell}(\mathbb{Z},\alpha)$, we denote by $\mathbb{P}^{\mathbf{x},\mathbf{y},\mathbf{z}}(\cdot)$ the law of $(X,Y,Z)$ given that $X_0=\mathbf{x}$, $Y_0=\mathbf{y}$, $Z_0=\mathbf{z}$, where $X$, $Y$ and $Z$ are independent simple random walks on $\text{Comb}_{\ell}(\mathbb{Z},\alpha)$. Similarly, we write $\mathbb{P}^\mathbf{x}$ when referring to a single simple random walk on the comb. Constants $C,c,c_1,c_2,\dots$ take values in $(0,\infty)$ and, unless explicitly specified, may change from line to line. Moreover, we use the notation $x\vee y\coloneqq\max\{x,y\}$ and $x\wedge y\coloneqq\min\{x,y\}$. Finally, we sometimes use a continuous variable $x$ where a discrete argument is required with the understanding that this should be replaced by $\lfloor x\rfloor$ or $\lceil x\rceil$, as appropriate, and similarly, we sometimes write $x\in [a,b]$ to mean $x\in [a,b]\cap \mathbb{N}_0$.

\paragraph{Structure of the paper.} In Section \ref{finite}, using a heat kernel estimate, we show that, when $\alpha>1$, three independent random walks collide only finitely often with probability one, i.e.\ we establish Theorem \ref{mainthm}(a). This is followed in Section \ref{heatkernelestimates} by a presentation of further heat kernel estimates. These are applied in Section \ref{infinite} to show that, when $\alpha\leq 1$, three independent random walks collide infinitely often with probability one; specifically, we prove Theorem \ref{mainthm}(b) and Theorem \ref{thmgrowth}. We conclude with an appendix where we recall elementary results that are used throughout the article. 

\section{Finitely many triple collisions occur when $\alpha> 1$}\label{finite}

The goal here is to show that, when $\alpha> 1$, the number of \textit{triple} collisions 
\begin{equation*}
   \left|\{n\in \mathbb{N}_0:X_n=Y_n=Z_n\}\right|
\end{equation*}
is finite with probability one. As noted in the introduction, to establish this property, we show the expected number of triple collisions is finite by applying heat kernel estimates for the comb graph. We define the heat kernel (or, alternatively, the transition density) by setting
\[p_{n}(\mathbf{x},\mathbf{y})\coloneqq\frac{\mathbb{P}^\mathbf{x}\left(X_n=\mathbf{y}\right)}{\mathrm{deg}(\mathbf{y})},\]
where $\mathrm{deg}(\mathbf{y})$ is the usual graph degree of $\mathbf{y}$. We highlight that $1\leq \mathrm{deg}(\mathbf{y})\leq 3$ for all $\mathbf{y}$, and so the above transition density is comparable to the transition probability $\mathbb{P}^\mathbf{x}(X_n=\mathbf{y})$. We also note that, since $\mathrm{deg}(\mathbf{y})$ gives an invariant measure for the reversible Markov chain $X$, the heat kernel is symmetric, i.e.\
$p_{n}(\mathbf{x},\mathbf{y})=p_{n}(\mathbf{y},\mathbf{x})$. We further define
\begin{equation}\label{vdef}
B(\x,r)\coloneqq\left\{\mathbf{y}\in\text{Comb}_{\ell}(\mathbb{Z},\alpha):\:d(\x,\mathbf{y})\leq r\right\},\qquad V(\x,r)\coloneqq|B(\x,r)|,
\end{equation}
where $d(\mathbf{x},\mathbf{y})$ is the usual shortest path graph distance between $\mathbf{x}$ and $\mathbf{y}$. Our first result provides an estimate on the probability of a triple collision in terms of the transition density.

\begin{lemma}\label{l21} For every $n\geq 1$, it holds that
   \[ \mathbb{P}^{\mathbf{0},\mathbf{0},\mathbf{0}}\left(X_n=Y_n=Z_n\right)\leq9 \sup_{\x\in B(\0,n)}p_{2\lfloor{n/2}\rfloor}(\mathbf{x},\mathbf{x})^2.\]
\end{lemma}
\begin{proof}
Since we are dealing with nearest-neighbour random walks, $X_n\in B(\0,n)$ for each $n$ (and similarly for $Y_n,Z_n$). Hence, recalling that the degree of each vertex is at most $3$, a straightforward calculation gives
\begin{eqnarray*}
    \mathbb{P}^{\mathbf{0},\mathbf{0},\mathbf{0}}\left(X_n=Y_n=Z_n\right)&=&\sum_{\x\in B(\0,n)} p_n(\0,\x)^3\mathrm{deg}(\mathbf{x})^3\\
    &\leq & 9\sup_{\x\in B(\0,n)} p_n(\0,\x) \sum_{\x\in B(\0,n)} p_n(\0,\x)^2\mathrm{deg}(\mathbf{x})\\
    &=& 9\sup_{\x\in B(\0,n)} p_n(\0,\x) \times p_{2n}(\0,\0),
\end{eqnarray*}
where for the last equality we have used that
\begin{align*}
    \sum_{\x\in B(\0,n)} p_n(\0,\x)^2\mathrm{deg}(\mathbf{x})&=\sum_{\x\in B(\0,n)}\mathbb{P}^{\0}(X_n=x)\mathbb{P}^{\0}(X_n=x)\mathrm{deg}(\mathbf{x})^{-1}\\
    &=\sum_{\x\in B(\0,n)}\mathbb{P}^{\0}(X_n=\x)\mathbb{P}^{\x}(X_n=\0)\mathrm{deg}(\mathbf{0})^{-1}\\
    &=\mathrm{deg}(\mathbf{0})^{-1}\mathbb{P}^{\0}(X_{2n}=\0)\\
    &=p_{2n}(\0,\0).
\end{align*}
Now observe that 
\begin{align*}
    p_n(\0,\x)&=\deg(\x)^{-1}\sum_{\y}^{}\mathbb{P}^{\0}(X_n=x,X_{\lfloor{n/2}\rfloor}=y)\\
    &=\deg(\x)^{-1}\sum_{\y}^{}\mathbb{P}^{\0}(X_{\lfloor{n/2}\rfloor}=y)\mathbb{P}^{\y}(X_{\lceil{n/2}\rceil}=x)\\
    &=\sum_{\y}^{}p_{\lfloor{n/2}\rfloor}(\0,\mathbf{y})p_{\lceil{n/2}\rceil}(\mathbf{y},\x)\mathrm{deg}(\mathbf{y}),
\end{align*}
and so, by Cauchy-Schwarz, we further have that
\begin{align}\label{cs}
 \nonumber p_n(\0,\x)&=\sum_{\mathbf{y}}p_{\lfloor{n/2}\rfloor}(\0,\mathbf{y})p_{\lceil{n/2}\rceil}(\mathbf{y},\x)\mathrm{deg}(\mathbf{y})\\
 \nonumber&\leq \left(\sum_{\mathbf{y}}p_{\lfloor{n/2}\rfloor}(\0,\mathbf{y})p_{\lfloor{n/2}\rfloor}(\y,\mathbf{0})\mathrm{deg}(\mathbf{y})\right)^{1/2}\left(\sum_{\mathbf{y}}p_{\lceil{n/2}\rceil}(\mathbf{x},\y)p_{\lceil{n/2}\rceil}(\mathbf{y},\x)\deg(\y)\right)^{1/2}\\
&= \sqrt{p_{2\lfloor{n/2}\rfloor}(\0,\0)p_{2\lceil{n/2}\rceil}(\x,\x)}.
\end{align}
Thus, by the monotonicity of the on-diagonal part of the transition density (i.e.\ $p_{2(n+1)}(\x,\x)\leq p_{2n}(\x,\x)$; see \cite[(15)]{CHLLT}, for example), it follows that
\[\mathbb{P}^{\mathbf{0},\mathbf{0},\mathbf{0}}\left(X_n=Y_n=Z_n\right)\leq 9\sup_{\x\in B(\0,n)}\sqrt{p_{2\lfloor{n/2}\rfloor}(\0,\0)p_{2\lceil{n/2}\rceil}(\x,\x)}  \times p_{2n}(\0,\0) \leq 9\sup_{\x\in B(\0,n)}p_{2\lfloor{n/2}\rfloor}(\mathbf{x},\mathbf{x})^2, \]
as desired.
\end{proof}

We next proceed to give an estimate on the size of the upper bound in the previous lemma.

\begin{lemma}\label{hku1} There exists a constant $C$ such that, for every $n\geq 2$, 
   \[\sup_{\x\in B(\0,n)}p_{2\lfloor{n/2}\rfloor}(\mathbf{x},\mathbf{x})\leq \frac{C}{n^{1/2}\log^{\alpha/2}(n)}.\]
\end{lemma}
\begin{proof}
Using known connections between volume, resistance and transition density estimates, the key to the proof will be establishing a good volume lower bound. In particular, in the proof of \cite[Lemma 11]{CHLLT}, it is shown that
\begin{equation}\label{hkbound}
    p_{2\lfloor{n/2}\rfloor}(\mathbf{x},\mathbf{x})\leq \frac{4r}{\lfloor{n/2}\rfloor}+\frac{2}{V(\x,r)},
\end{equation}
for any $r\geq 1$, with $V(\x,r)$ being the volume of the ball with respect to the resistance metric of the graph in question. Since we are dealing with a graph tree, which means the resistance distance is identical to the graph distance $d$,
we can consider $V(\x,r)$ to be defined as at \eqref{vdef}. With a suitable lower bound on the volume in place, we will optimise over $r$ to obtain the result.

In the following, we will suppose $n_0$ is chosen suitably large so that $\frac14 n^{1/4}\geq \log^\alpha(n)$ for all $n\geq n_0$, and restrict ourselves to consideration of $n$ in this range. (Smaller values of $n$ can then be dealt with by simply adjusting $C$ as necessary.) Moreover, we will suppose $\x=(x,u)\in B(\0,n)$ is such that $x\geq0$ (the other case is dealt with similarly), and that $r\geq n^{1/4}$. From the geometry of the graph, it is possible to check that
\[V(\x,r)\geq \sum_{y=x+1}^{x+r-u}\log^\alpha(y)\wedge (r-u-y+x).\]
Indeed, to see this, first of all note that for a site $(y,z)$ to be contained in $V(\x,r)$ it is necessary that $x-r+u\leq y\leq x+r-u$, and hence, for a lower bound, we can indeed restrict to $x+1\leq y\leq x+r-u$. Next, observe that the number of vertices located on the tooth at $(y,0)$ is, by definition, $\log^{\alpha}(y)$; however, when $y$ becomes `too large', not all of the tooth at $y$ is contained in $V(\x,r)$ -- some of the sites located at the top of the tooth are farther than $r$ from $(x,u)$. More precisely, consider vertex $(y,z)$; for it to be in $V(\x,r)$, it must be that $z+(y-x)+u\leq r$, i.e. $z\leq r-u-y+x$. Therefore, only $\log^{\alpha}(y)\wedge (r-u-y+x)$ sites on the tooth at $(y,0)$ are actually present in $V(\x,r)$, thus establishing the desired volume lower bound.

Since the tooth at $(x,0)$ has height $\log^{\alpha}(x)\geq u$ and $\x=(x,u)\in B(\0,n)$ (which in particular implies that $x\leq n$), we obtain
\begin{equation}\label{primastima}
    r-u\geq r-\log^\alpha(x)\geq r-\log^\alpha(n)\geq r-\frac14 n^{1/4}\geq \frac34 r.
\end{equation}
Hence, if $y-x\leq 3r/8$, then 
\begin{equation}\label{secondastima}
    r-u-y+x \geq \frac34 r -\frac38 r =\frac38 r.
\end{equation}
Moreover, for $y-x\in [r/4,3r/8]$, it holds that (as $x\geq 0$,) 
\[\log^\alpha(y)\geq \log^\alpha(x+r/4)\geq \log^\alpha(r/4),\]
and also, (as $x\leq n$,)
\begin{align}\label{terzastima}
    \nonumber\log^\alpha(y)&\leq \log^\alpha(x+3r/8)\leq \log^\alpha(n+3r/8)\leq \frac14 \left(n+\frac{3}{8}r\right)^{1/4}\\
    &\leq \frac{1}{4}n^{1/4}+\left(\frac{3}{8}r\right)^{1/4}\leq \frac{1}{4}r+\left(\frac{3}{8}r\right)^{1/4}\leq \frac{3}{8}r,
\end{align}
where to ensure the final inequality we increase $n_0$ if necessary. Therefore we arrive at
\begin{align*}
   V(\x,r)&\geq  \sum_{y=x+1}^{x+r-u}\log^\alpha(y)\wedge (r-u-y+x)\geq \sum_{y=x+r/4}^{x+3r/8}\log^\alpha(y)\wedge (r-u-y+x)\\
   &\geq \sum_{y=x+r/4}^{x+3r/8}\log^\alpha(y)\geq  \frac{1}{8}r\log^\alpha\left(\frac14 r\right),
\end{align*}
where the second inequality uses (\ref{primastima}) together with the fact that $r/4\geq 1$ for large $n$, whereas the third inequality follows from a combination of (\ref{secondastima}) and (\ref {terzastima}).

Applying the bound of the previous paragraph at $r=n^{1/2}\log^{-\alpha/2}(n)$, which is clearly larger than $n^{1/4}$ when $n$ is suitably large, we obtain from \eqref{hkbound} that
\[p_{2\lfloor{n/2}\rfloor}(\mathbf{x},\mathbf{x})\leq C\left(\frac{n^{1/2}\log^{-\alpha/2}(n)}{n}+\frac{8}{n^{1/2}\log^{-\alpha/2}(n)\log^\alpha(n)}\right)\leq\frac{C}{n^{1/2}\log^{\alpha/2}(n)}.\]
Since the constant $C$ is uniform over $\x\in B(\0,n)$, we are done.
\end{proof}

Since $1/n\log^\alpha(n)$ is summable for $\alpha>1$, combining the two previous lemmas readily yields the claim of Theorem \ref{mainthm}(a). Indeed, we obtain the stronger claim that the expected number of triple collisions is finite:
\begin{align*}
    \mathbb{E}^{\0,\0,\0}\left[\left|\{n\in \mathbb{N}_0:X_n=Y_n=Z_n\}\right|\right]&=\sum_{n\in \mathbb{N}_0}^{}\mathbb{P}^{\0,\0,\0}(X_n=Y_n=Z_n)\\
    &\leq 2+9\sum_{n\geq 2}^{}\sup_{\x\in B(\0,n)}p_{2\lfloor{n/2 \rfloor}}(\x,\x)^2\\
    &\leq 2+9C\sum_{n\geq 2}^{}\frac{1}{n\log^{\alpha}(n)}\\
    &<\infty.
\end{align*}

\begin{remark}\label{4collisionrem}
    The same argument will give that the number of quadruple (or more) collisions is finite for any infinite, connected graph of bounded degree. Indeed, for any infinite, connected graph, the volume of the (effective resistance) ball of radius $r$ is at least $r$. Hence, taking $r=n^{1/2}$ in \eqref{hkbound} yields $p_{2\lfloor n/2\rfloor}(\mathbf{x},\mathbf{x})\leq Cn^{-1/2}$ for any $n\geq 1$ and vertex $\mathbf{x}$ in the graph in question, where $C$ is some universal constant. Hence, from a straightforward adaptation of Lemma \ref{l21}, we obtain that the probability of a quadruple collision at time $n$ is bounded above by $Cn^{-3/2}$. Since the latter quantity is summable, we obtain the expected number of such collisions is finite.
\end{remark}

\section{Further heat kernel estimates}\label{heatkernelestimates}

In order to establish the infinite triple collision property when $\alpha\leq 1$, we will apply various further heat kernel (transition density) estimates. The aim of this section is to derive these, which we do by applying standard arguments. For introductory background on the techniques, see \cite{Barbook, Kuma}, for example. We start with an upper bound that is an adaptation of Lemma \ref{hku1}. To state the result, we introduce the notation:
\begin{equation}\label{vndef}
\mathbb{V}_N\coloneqq\left\{\x=(x,u)\in \text{Comb}_{\ell}(\mathbb{Z},\alpha):\:|x|\leq N\right\}.
\end{equation}

\begin{lemma}\label{hku2} There exists a constant $c$ such that, for every $N\geq 2$, $\mathbf{x},\mathbf{y}\in \mathbb{V}_N\backslash\mathbb{V}_{N/4}$ and $n\geq 1$,
\[p_{n}(\mathbf{x},\mathbf{y})\leq \begin{cases}
    cn^{-1/2},&\mbox{if }n< 16\log^{3\alpha}(N),\\
    cn^{-1/2}\log^{-\alpha/2}(N),&\mbox{if }n\geq 16\log^{3\alpha}(N).\\
\end{cases}\]
\end{lemma}
\begin{proof}
By applying Cauchy-Schwarz as at \eqref{cs}, it will suffice to consider the case $\mathbf{x}=\mathbf{y}$ and $n$ even. As in the proof of Lemma \ref{hku1}, the key estimate we will apply is \eqref{hkbound}. This requires a lower bound on the volume. Similarly to the comment in Remark \ref{4collisionrem}, it is clear that, for any $r\geq1$,
\[V(\mathbf{x},r)\geq r,\]
where we define $V(\x,r)$ as in \eqref{vdef}. Hence, setting $r=n^{1/2}$, we obtain from \eqref{hkbound} that
\[p_n(\mathbf{x},\mathbf{x})\leq Cn^{-1/2}\]
for any $n\geq 1$ and any $\mathbf{x}\in \mathrm{Comb}_\ell(\mathbb{Z},\alpha)$. (For odd $n$, the left-hand side above is zero.) Next, suppose $\mathbf{x}=(x,u)$ is such that $x\in\{\lceil{N/4}\rceil,\lceil{N/4}\rceil+1,\dots,N\}$; the case when $x<0$ can be dealt with similarly. Arguing as in the proof of Lemma \ref{hku1}, it is further possible to check that, for $N\geq 16$ and $r\geq 4\log^\alpha(N)$, (in which case $r-u\geq 4\log^{\alpha}(N)-u\geq 3\log^{\alpha}(N)$ and the sum below is non-empty,)
\begin{equation}\label{vlb}
V(\mathbf{x},r)\geq \sum_{y=x+1}^{x+r-u}\log^\alpha(y)\wedge(r-u-y+x)\geq \sum_{y=x+\log^\alpha(N)}^{x+r-u-\log^\alpha(N)}\log^\alpha(N/4)\geq \frac{r}{4}\log^\alpha(N/4)\geq  \frac{r}{8}\log^\alpha(N).
\end{equation}
Note that to deduce the second inequality, we apply the assumption that $x\geq N/4$, and to deduce the final one, we use that $N\geq 16$. If $n\geq 16\log^{3\alpha}(N)$ and $r=n^{1/2}\log^{-\alpha/2}(N)$, then $r\geq 4\log^{\alpha}(N)$. Hence the above volume bound applies, and we deduce from \eqref{hkbound} that
\[p_n(\mathbf{x},\mathbf{x})\leq \frac{8n^{1/2}\log^{-\alpha/2}(N)}{n}+\frac{2}{cn^{1/2}\log^{\alpha/2}(N)}=\frac{c}{n^{1/2}\log^{\alpha/2}(N)},\]
as desired. This completes the proof for $N\geq 16$. The case when $2\leq N<16$ can be dealt with by adjusting the constant $c$ in the upper bound as necessary.
\end{proof}

We next provide a lower bound for the heat kernel. The form is similar to that of the upper bound, but we need to be more careful about the relationship between the time $n$ and the distance between the points $\mathbf{x}$ and $\mathbf{y}$ being considered. Moreover, for our later purposes, it will be convenient to consider the heat kernel of the random walk killed on exiting a certain spatial region. In particular, for a subset $B\subseteq \mathrm{Comb}_\ell(\mathbb{Z},\alpha)$, we set
\begin{equation}\label{killedtddef}
    p_n^B(\mathbf{x},\mathbf{y})\coloneqq\frac{\mathbb{P}^\mathbf{y}\left(X_n=\mathbf{y},\:\tau_B>n\right)}{\mathrm{deg}(\mathbf{y})},
\end{equation}
where $\tau_B$ is the exit time of $B$ by $X$, i.e.,
\[\tau_{B}\coloneqq\inf\left\{n\geq 0:\:X_n\not\in  B\right\}.\]

\begin{lemma}\label{lowerboundheatkernel}
    Fix a finite constant $C\geq 1$. There exist constants $c_1,c_2>0$ and $h\geq 1$ for which the following holds: if $N\geq 2$, $\mathbf{x}\in \mathbb{V}_N$, $ \mathbf{y}\in\mathbb{V}_N\backslash\mathbb{V}_{N/4}$ and $n\in[2C+16\log^{3\alpha}(N),CN^2\log^\alpha(N)]$ are such that $d(\mathbf{x},\mathbf{y})-n$ is even and $d(\mathbf{x},\mathbf{y})\leq c_1n^{1/2}\log^{-\alpha/2}(N)$, then
    \[p^{\mathbb{V}_{hN}}_n(\mathbf{x},\mathbf{y})\geq \frac{c_2}{n^{1/2}\log^{\alpha/2}(N)}.\]
    \end{lemma}
\begin{proof}
For $\mathbf{x}=(x,u)\in\mathbb{V}_N\backslash\mathbb{V}_{N/4}$ and $r\geq 1$, write $\tau_{\mathbf{x},r}$ for $\tau_{B(\mathbf{x},r)}$. It is a standard result (see \cite[Lemma 4.1.1]{Kuma}, for example) that
\begin{equation}\label{etu}
\sup_{\mathbf{y}\in B(\mathbf{x},r)}\mathbb{E}^\mathbf{y}\left(\tau_{\mathbf{x},r}\right)\leq 12rV(\mathbf{x},r).
\end{equation}
Precisely, the statement of \cite[Lemma 4.1.1]{Kuma} takes in the upper bound the product of (the supremum over $\mathbf{y}$ of) the effective resistance from $\mathbf{y}$ to $B(\mathbf{x},r)^c$ when the graph is considered as an electrical network with unit conductors placed along each edge, which is bounded above by $2(r+1)\leq 4r$ in our case, and the sums of degrees of vertices in $B(\mathbf{x},r)$. Since all vertices have degree less than 3, we obtain the above bound. Now, take $n$ as in the statement of the lemma, and set $r=c_0n^{1/2}\log^{-\alpha/2}(N)$, where $c_0$ is a constant that will be chosen later. We then have that $r\leq c_0 C^{1/2}N$. Hence, since there are at most $2r$ teeth in $B(\x,r)$ (other than the one on which $\x$ is located) and each one of them contains no more than $\log^{\alpha}(x+r)\leq \log^{\alpha}(N+r)$ vertices, we obtain
\begin{equation}\label{vub}
V(\mathbf{x},r)\leq (2r+1)\log^\alpha(N+r)\leq c_1 r\log^\alpha(N), 
\end{equation}
where $c_1$ depends only upon $c_0$ and $C$. In particular, together with \eqref{etu} this yields
\begin{equation}\label{upper}
\sup_{\mathbf{y}\in B(\mathbf{x},r)}\mathbb{E}^\mathbf{y}\left(\tau_{\mathbf{x},r}\right)\leq 12c_1r^2\log^\alpha(N).
\end{equation}
We also need a corresponding lower bound for $\mathbb{E}^\mathbf{x}(\tau_{\mathbf{x},r})$. To this end, we note that
\begin{equation}\label{etlb}
    \mathbb{E}^\mathbf{x}\left(\tau_{\mathbf{x},r}\right)=\sum_{\mathbf{y}\in B(\mathbf{x},r)}g(\mathbf{y})\mathrm{deg}(\mathbf{y}),
\end{equation}
where $g(\mathbf{y})$ is the occupation density of $X$ when started from $\x$ and run up to $\tau_{\mathbf{x},r}$. In particular, we have from \cite[Proposition 4.1.4]{Kuma} that
\[2g(\mathbf{y})=R\left(\mathbf{x},B(\mathbf{x},r)^c\right)+R\left(\mathbf{y},B(\mathbf{x},r)^c\right)-R_{B(\mathbf{x},r)^c}\left(\mathbf{x},\mathbf{y}\right),\]
where $R(A,B)$ represents the effective resistance between subsets (or vertices) $A$ and $B$ (again, when unit conductors are placed along edges of the graph), and the final term represents the effective resistance between $\mathbf{x}$ and $\mathbf{y}$ when $B(\mathbf{x},r)^c$ is `fused', see \cite[Section 2]{Kuma} for background on effective resistance and \cite[(4.9)]{Kuma} specifically for the definition of the fused effective resistance. We will give an estimate on $g(\mathbf{y})$ for $\mathbf{y}\in B(\mathbf{x},r/16)$. First, as above, suppose $r=c_0n^{1/2}\log^{-\alpha/2}(N)$, and note that this implies $r\geq c_0 4 \log^\alpha(N)$. Moreover note that, for $x< y\leq x+r/4$, the distance from $\mathbf{x}=(x,u)$ to the endpoint of the tooth at $y$ is given by
\[u+y-x+\log^\alpha(y)\leq \log^\alpha(N)+r/4+\log^\alpha(N+r/4)\leq 2\log^\alpha(N)+r/2\leq r\left(\frac{1}{2}+\frac{1}{2c_0}\right)\leq r,\]
where we assume that $c_0$ is large enough so that $c_0\geq e\log^{-\alpha}(2)$, which implies $r/4\geq e$ (so $\log^\alpha(r/4)\leq r/4$) and also gives the final inequality above. As a consequence of the above estimate (and a similar one when $y\leq x$), we deduce that any path from $\mathbf{x}$ to $B(\mathbf{x},r)^c$ has to cross the arc from $(x,0)$ to $(x+r/4,0)$ or that from $(x,0)$ to $(x-r/4,0)$. It readily follows that
\[R(\mathbf{x},B(\mathbf{x},r)^c)\geq r/8.\]
By the triangle inequality for resistances and the fact that $R_{B(\mathbf{x},r)^c}(\mathbf{x},\mathbf{y})\leq d(\mathbf{x},\mathbf{y})$, one further has that, if $\mathbf{y}\in B(\mathbf{x},r/16)$, then
\[R\left(\mathbf{y},B(\mathbf{x},r)^c\right)\geq R\left(\mathbf{x},B(\mathbf{x},r)^c\right)-R_{B(\mathbf{x},r)^c}\left(\mathbf{x},\mathbf{y}\right)\geq r/8-r/16=r/16.\]
Hence, for $\mathbf{y}\in B(\mathbf{x},r/16)$, we have proved that
\[2g(\mathbf{y})\geq r/8+r/16-r/16=r/8.\]
Returning to \eqref{etlb}, this implies
\[\mathbb{E}^\mathbf{x}\left(\tau_{\mathbf{x},r}\right)\geq \frac{r}{16}V(\mathbf{x},r/16).\]
Increasing $c_0$ if necessary to ensure $c_0\geq 16$ and recalling that $r\ge c_04\log^{\alpha}(N)$, we find that $r/16\geq 4\log^\alpha(N)$, and so, for $N\geq 16$, we can apply \eqref{vlb} to give that
\begin{equation}\label{lower}
\mathbb{E}^\mathbf{x}\left(\tau_{\mathbf{x},r}\right)\geq \frac{1}{2048}r^2\log^\alpha(N).
\end{equation}
Continuing to suppose $N\geq 16$ and combining the bounds at \eqref{upper} and \eqref{lower}, we have by the standard argument involving an application of the Markov property at time $t\geq 1$ (cf.\ the proof of \cite[Proposition 4.4.3]{Kuma}) that
\[\frac{1}{2048}r^2\log^\alpha(N)\leq \mathbb{E}^\mathbf{x}\left(\tau_{\mathbf{x},r}\right)\leq t+ 
\mathbb{P}^\mathbf{x}\left(\tau_{\mathbf{x},r}>t\right)\sup_{\mathbf{y}\in B(\mathbf{x},r)}\mathbb{E}^\mathbf{y}\left(\tau_{\mathbf{x},r}\right)\leq t+\mathbb{P}^\mathbf{x}\left(\tau_{\mathbf{x},r}> t\right)12c_1r^2\log^\alpha(N).\]
Specifically, we obtain that, if $t\leq r^2\log^\alpha(N)/4096$, then
\begin{equation}\label{exitprob}
\mathbb{P}^\mathbf{x}\left(\tau_{\mathbf{x},r}> t\right)\geq c_2
\end{equation}
for some constant $c_2$ depending on $c_1$. To complete our argument for estimating from below the on-diagonal part of the heat kernel, we note that, similarly to \cite[Proposition 4.3.4]{Kuma}, for all $n$ and $r$,
\[\mathbb{P}^\mathbf{x}\left(\tau_{\mathbf{x},r}>n\right)^2=\left(\sum_{\mathbf{y}\in B(\mathbf{x},r)}p_n^{B(\mathbf{x},r)}(\mathbf{x},\mathbf{y})\mathrm{deg}(\mathbf{y})\right)^2\leq  3p_{2n}^{B(\mathbf{x},r)}(\mathbf{x},\mathbf{x})V(\mathbf{x},r),\]
where the second inequality here is derived using Cauchy-Schwarz. Taking our choice of $r=c_0n^{1/2}\log^{-\alpha/2}(N)$, we have that $n=c_0^{-2}r^2\log^\alpha (N)$, which is less than $r^2\log^\alpha(N)/4096$ if $c_0\geq 64$. Hence, combining the previous bound, \eqref{vub} and \eqref{exitprob} yields
\begin{equation}\label{odlower}
p_{2n}^{B(\mathbf{x},r)}(\mathbf{x},\mathbf{x})\geq \frac{c_3}{r\log^\alpha(N)}=\frac{c_4}{n^{1/2}\log^{\alpha/2}(N)}.
\end{equation}
Collecting the constraints, we highlight that this argument is valid if $N\geq 16$ and $c_0\geq 64\vee16\vee e\log^{-\alpha}(2)=64$. The result is readily extended to $2\leq N<15$ by reducing $c_4$ as required.

Finally, we deal with the case when we have $\mathbf{x}\in \mathbb{V}_N$ and $\mathbf{y}\in\mathbb{V}_N\backslash\mathbb{V}_{N/4}$ with $\mathbf{x}\neq \mathbf{y}$. To do this, we start by observing that, similarly to the argument used in the proof of \cite[Lemma 3.4]{Chenchen}, if  $d(\mathbf{x},\mathbf{y})-n$ is even and $B\subseteq \mathrm{Comb}_\ell(\mathbb{Z},\alpha)$, then
\begin{equation}\label{ilb}
p_n^B(\mathbf{x},\mathbf{y})=\sum_{k=0}^{\lfloor n/2\rfloor}\mathbb{P}^\mathbf{x}\left(\tau_B>n-2k,\:\tau_{\{\mathbf{y}\}^c}=n-2k\right)p^B_{2k}(\mathbf{y},\mathbf{y})\geq\mathbb{P}^\mathbf{x}\left(\tau_{\{\mathbf{y}\}^c}\leq n\wedge \tau_B\right)p^B_{2\lfloor n/2\rfloor}(\mathbf{y},\mathbf{y}),
\end{equation}
where the equality is obtained by conditioning on the time at which $X$ first hits $\y$, $\tau_{\{\mathbf{y}\}^c}$, and the inequality holds because $p^B_{2k}(\mathbf{y},\mathbf{y})$ is decreasing in $k$ (see \cite[(3.3)]{BPS}, for example). Now, take
\[B=B(\mathbf{x},(c_0+c_1)n^{1/2}\log^{-\alpha/2}(N))\]
where $c_0$ is as in the first part of the proof and $c_1$ is as in the statement of the lemma. We then have that
\begin{align*}
    \mathbb{P}^\mathbf{x}\left(\tau_{\{\mathbf{y}\}^c}> n\wedge \tau_B\right)&\leq\mathbb{P}^\mathbf{x}\left(\tau_{\{\mathbf{y}\}^c}\wedge\tau_B >n\right)+\mathbb{P}^\mathbf{x}\left(\tau_{\{\mathbf{y}\}^c}>\tau_B\right).
\end{align*}
By the commute time identity (see \cite[Proposition 10.6]{LPW}, for example), the first of these terms is bounded above as follows:
\[\mathbb{P}^\mathbf{x}\left(\tau_{\{\mathbf{y}\}^c}\wedge\tau_B >n\right)\leq \frac{3V(\mathbf{x},(c_0+c_1)n^{1/2}\log^{-\alpha/2}(N))d(\mathbf{x},\mathbf{\mathbf{y}})}{n},\]
where we use that $d(\mathbf{x},\mathbf{\mathbf{y}})$ is an upper bound for the effective resistance from $\mathbf{x}$ to $\{\mathbf{y}\}\cup B^c$. Moreover, by \cite[Exercise 2.36]{LP}, for example, the second term satisfies
\[\mathbb{P}^\mathbf{x}\left(\tau_{\{\mathbf{y}\}^c}>\tau_B\right)\leq  \frac{d(\mathbf{x},\mathbf{\mathbf{y}})}{R(\mathbf{x},B^c)}.\]
Since we assume $d(\mathbf{x},\mathbf{y})\leq c_1n^{1/2}\log^{-\alpha/2}(N)$, with $n\in[2C+16\log^{3\alpha}(N),CN^2\log^\alpha(N)]$, it is elementary to check that
\begin{align*}
\lefteqn{V(\mathbf{x},(c_0+c_1)n^{1/2}\log^{-\alpha/2}(N))d(\mathbf{x},\mathbf{y})}\\
&\leq 4(c_0+c_1)c_1n\log^{-\alpha}(N)\log^\alpha(N+(c_0+c_1)C^{1/2}N)\\
&\leq4(c_0+c_1)c_1n\left(1+\frac{\log^\alpha(1+(c_0+c_1)C^{1/2})}{\log^\alpha(2)}\right).
\end{align*}
Additionally, we have from above that
\[R(\mathbf{x},B^c)\geq R(\mathbf{x},B(\mathbf{x},c_0n^{1/2}\log^{-\alpha/2}(N))^c)\geq \frac18c_0n^{1/2}\log^{-\alpha/2}(N).\]
Putting together the various estimates, we find that
\begin{equation}\label{c0choice}
\mathbb{P}^\mathbf{x}\left(\tau_{\{\mathbf{y}\}^c}> n\wedge \tau_B\right)\leq 12(c_0+c_1)c_1\left(1+\frac{\log^\alpha(1+(c_0+c_1)C^{1/2})}{\log^\alpha(2)}\right)+\frac{8c_1}{c_0}.
\end{equation}
In particular, by choosing $c_1$ suitably small, this probability can be made smaller than $1/2$, and we obtain from \eqref{ilb} that
\[p_n^B(\mathbf{x},\mathbf{y})\geq \frac12p^B_{2\lfloor n/2\rfloor}(\mathbf{y},\mathbf{y})\geq  \frac12p^B_{2n}(\mathbf{y},\mathbf{y})\geq  \frac12p^{B(\mathbf{y},c_0n^{1/2}\log^{-\alpha/2}(N))}_{2n}(\mathbf{y},\mathbf{y}),\]
where for the final inequality we use that $B\supseteq B(\mathbf{y},c_0n^{1/2}\log^{-\alpha/2}(N))$. Moreover, recalling that $\x\in \mathbb{V}_N$ by assumption, we have that $B(\mathbf{x},r)\subseteq \mathbb{V}_{N+r}\subseteq \mathbb{V}_{(1+c_0C^{1/2})N}$, which we can ensure is contained within $\mathbb{V}_{hN}$ by taking $h\geq 1+c_0C^{1/2}$. So, we conclude from the above bound and \eqref{odlower} that
\[p_n^{\mathbb{V}_{hN}}(\mathbf{x},\mathbf{y})\geq p_n^B(\mathbf{x},\mathbf{y}) \geq \frac12p^{B(\mathbf{y},c_0n^{1/2}\log^{-\alpha/2}(N))}_{2n}(\mathbf{y},\mathbf{y}) \geq\frac{c_4}{2n^{1/2}\log^{\alpha/2}(N)},\]
which confirms the result.
\end{proof}

Before concluding this section, we give an adaptation of the previous result that will also be useful for us later. We note the bound runs over a restricted range of $n$, but applies to all $\mathbf{x}\in \mathbb{V}_N$ and $\y\in \mathbb{V}_N\backslash\mathbb{V}_{N/4}$.

\begin{corollary}\label{cormainHKLB}
    There exist constants $c_1,c_2,c_3>0$ with $c_1<c_2$ and $h\geq 1$ for which the following holds: if $N\geq 2$, $\mathbf{x}\in \mathbb{V}_N$, $\y\in \mathbb{V}_N\backslash\mathbb{V}_{N/4}$ and $n\in [c_1N^2\log^{\alpha}(N),c_2N^2\log^{\alpha}(N)]$ are such that $d(\x,\y)-n$ is even, then 
    \[p^{\mathbb{V}_{hN}}_n(\x,\y)\geq \frac{c_3}{N\log^{\alpha}(N)}.\]
\end{corollary}
\begin{proof}
Observe that, in the proof of Lemma \ref{lowerboundheatkernel}, $c_0$ could be chosen independently of $C$, and the requirement on $c_1$ was that it ensured the upper bound at \eqref{c0choice} was smaller than $1/2$. In particular, given any $c_0\geq 64$ and $\varepsilon\in (0,1/2)$, the latter requirement is met by taking $c_1=C^{-\varepsilon}$ with $C$ suitably large. Suppose this is the case. Now, if $\mathbf{x}\in \mathbb{V}_{ N}$ and $\y\in \mathbb{V}_{N\backslash\mathbb{V}_{N/4}}$, then
  \[d(\x,\y)\leq 1+2 N+2\log^\alpha( N)\leq 3 N\]
  for $N$ suitably large. Hence, if $n\in [(C/2)N^2\log^\alpha(N),CN^2\log^\alpha(N)]$, then (for sufficiently large $C$)
  \[d(\x,\y)\leq c_1\left(\frac{C}{2}\right)^{1/2}N\leq c_1n^{1/2}\log^{-\alpha/2}(N).\]
Since $n\geq(C/2) N^2\log^\alpha(N) \geq 2C+16\log^{3\alpha}(N)$ (again, for $N$ suitably large), we thus obtain from Lemma \ref{lowerboundheatkernel} that, if $d(\x,\y)-n$ is even, then 
    \[p^{\mathbb{V}_{hN}}_n(\mathbf{x},\mathbf{y})\geq \frac{c_2}{n^{1/2}\log^{\alpha/2}(N)}\geq \frac{c_3}{N\log^\alpha(N)},\] 
    which gives us the result for large $N$. For small $N$, there are only a finite number of cases to consider, and we can ensure the bound is met by adjusting the constants as necessary.
\end{proof}

\section{Infinitely many triple collisions occur when $\alpha\leq 1$}\label{infinite}

Recall that $(X_n)_{n\geq0}, (Y_n)_{n\geq0}$ and $(Z_n)_{n\geq0}$ are independent simple random walks on $\text{Comb}_{\ell}(\mathbb{Z},\alpha)$ started at the origin $\0=(0,0)$.
The goal here is to show that, when $\alpha\leq 1$, the number of \textit{triple} collisions 
\begin{equation*}
   \left|\{n\in \mathbb{N}_0:X_n=Y_n=Z_n\}\right|
\end{equation*}
is infinite with probability one and, to do so, we take inspiration from \cite{Chenchen} where the authors used the second moment method to establish that two random walks meet infinitely many times almost-surely before leaving `growing boxes' covering the underlying graph. (We remark that even though we apply the second moment method along the lines of \cite{Chenchen}, the heat-kernel estimates we depend upon for the first and second moment computations are novel and provide a possibly more robust approach to the collision problem.) 

We will start by supposing a bound on the probability that the three walkers meet before exiting a certain spatial region. In particular, for any $N\in \mathbb{N}$, recalling the definition of $\mathbb{V}_N$ from \eqref{vndef}, we set $\theta_{N}\coloneqq \theta^X_{N}\wedge \theta^Y_{N}\wedge \theta^Z_{N}$, where $\theta^X_{N}$ (respectively $\theta^Y_{N}$, $\theta^Z_{N}$) is the first time at which $X$ (respectively $Y$, $Z$) leaves $\mathbb{V}_N$, that is
\[\theta^X_{N}\coloneqq \inf\left\{n\in \mathbb{N}_0:X_n\notin \mathbb{V}_N\right\}\]
(and similarly for $\theta^Y_{N},\theta^Z_{N}$). Assume for the moment that there exists an integer $h\geq 2$ and constant $c>0$ such that, for all large enough $N$,
\begin{equation}\label{mainLB}
\mathbb{P}^{\x,\y,\z}\left(\exists n\in [0,\theta_{hN}): X_n=Y_n=Z_n\right)\geq\frac{c}{L(N)}
\end{equation}
for every $\x,\y,\z\in \mathbb{V}_N$ at even distance, where $L : \mathbb{N}\rightarrow (0,\infty)$ is defined by setting
\[L(n)\coloneqq\left\{
  \begin{array}{ll}
    \log(n), & \hbox{if $0<\alpha<1$;} \\
    \log(n)\log\log(n), & \hbox{if $\alpha=1$.}
  \end{array}
\right.\]
(Of course, the condition that the starting distances of the three walkers from each other must all be even is necessary for the three walkers to have any chance of meeting; even if we do not say so explicitly, we will always assume that this is the case.) Then, for $m\in\mathbb{N}$, denote by $\mathcal{H}_m$ the event that $X_n=Y_n=Z_n$ at some time $n$ between $\theta_{h^m}$ and $\theta_{h^{m+1}}$, i.e.,
\[\mathcal{H}_m\coloneqq \left\{\exists n\in [\theta_{h^m},\theta_{h^{m+1}}):\:X_n=Y_n=Z_n\right\}.\]
By the strong Markov property and (\ref{mainLB}), we obtain that, almost-surely, for all large $m$,
\begin{align*}
\lefteqn{\mathbb{P}^{\0,\0,\0}\left(\mathcal{H}_{m}\mid X_k,Y_k,Z_k,\:0\leq k\leq \theta_{h^m}\right)}\\
&=\mathbb{P}^{X_{\theta_{h^m}},Y_{\theta_{h^m}},Z_{\theta_{h^m}}}\left(\exists n\in [0,\theta_{h^{m+1}}): X_n=Y_n=Z_n\right)\\
&\geq \frac{c}{m^{}\log(m)},
\end{align*}
where $c>0$ is a deterministic constant (that might be different to the one in \eqref{mainLB}). Since the series $\sum_{m\geq2}1/(m\log(m))$ diverges, a conditional version of the Borel-Cantelli lemma (see Lemma \ref{BCL} in the appendix) then yields
\begin{equation}\label{bcapp}
\mathbb{P}^{\0,\0,\0}(\mathcal{H}_m \text{ infinitely often})=1.
\end{equation}
The desired result follows immediately after noticing that
\[\mathbb{P}^{\0,\0,\0}\left(\left|\{n\in \mathbb{N}_0:X_n=Y_n=Z_n\}\right|=\infty\right)\geq \mathbb{P}^{\0,\0,\0}(\mathcal{H}_m \text{ infinitely often}).\]
Thus, in order to show that three independent simple random walks on $\text{Comb}_{\ell}(\mathbb{Z},\alpha)$ collide infinitely many times with probability one when $\alpha\leq 1$, as is required to confirm Theorem \ref{mainthm}(b), it will suffice to establish (\ref{mainLB}) for some integer $h\geq 2$; we will do this in Section \ref{firstM} below. (The particular value of $h$ will be chosen in order to apply the heat kernel estimates of the previous section.) In the remainder of this section, we will go on to show a \textit{stronger} statement; more precisely, we will give an estimate on the growth rate of the number of (triple) collisions in $\mathbb{V}_{hN}$ (and, similarly, up to time $N$) in both regimes $\alpha <1$ and $\alpha=1$. As per Theorem \ref{thmgrowth}, we will show that the number of meetings diverges with $N$ at least as fast as $\log^{1-\alpha}(N)$ for $\alpha<1$ and at least as fast as $\log\log(N)$ when $\alpha=1$. 

\subsection{Counting the number of triple collisions in $\mathbb{V}_{hN}$}

The goal here is to prove the results described in the above discussion. Towards this end, we will show that, when $\alpha<1$, the number of triple collisions prior to $\theta_{hN}$ (which we recall is the first time at which one of the walkers leaves $\mathbb{V}_{hN}$) is at least of order $\log^{1-\alpha}(N)$ with probability bounded from below by $c/\log(N)$, whereas, when $\alpha=1$, the number of triple collisions prior to $\theta_{hN}$ is at least of order $\log\log(N)$ with probability bounded from below by $c/\log(N)\log\log(N)$. We establish these facts by means of a two-step procedure, which is based on an application of the second moment method (see Lemma \ref{PZI} in the Appendix). In particular, in Section \ref{firstM}, we show that there is at least one meeting before time $\theta_{hN}$ with probability at least $c/L(N)$; this is already enough to establish the lower bound at (\ref{mainLB}) (and thus, as noted above, Theorem \ref{mainthm}(b)). Subsequently, in Section \ref{subseqM}, we show that, once the three random walkers have met for the first time, they are bound to meet with constant, positive probability a divergent number of times (whose divergence rate depends on the value of $\alpha$, as anticipated earlier) before time $\theta_{hN}$. Finally, in Section \ref{finalsec}, we establish Theorem \ref{thmgrowth}.

\subsubsection{The first meeting}\label{firstM}

By Corollary \ref{cormainHKLB}, we know that there are positive constants $c_1,c_2,c_3>0$ with $c_1<c_2$ and $h\geq 1$ such that, if $N\geq 2$, $\mathbf{x}\in \mathbb{V}_N$, $\y\in \mathbb{V}_N\backslash\mathbb{V}_{N/4}$ and $n\in [c_1N^2\log^{\alpha}(N),c_2N^2\log^{\alpha}(N)]$ satisfy $d(\x,\y)-n$ is even, then 
\begin{equation}\label{jio}
p^{\mathbb{V}_{hN}}_n(\x,\y)\geq \frac{c_3}{N\log^{\alpha}(N)}.
\end{equation}
We use these constants to define, for a given $\varepsilon \in (0,(1/3)\wedge (1-c_1/c_2))$,
\begin{equation}\label{t1def}
    T_1=T_1(N)\coloneqq \lfloor c_2(1-\varepsilon) N^2 \log^{\alpha}(N)\rfloor,
\end{equation}
increase $h$ if necessary to be an integer that is greater than $2$ (note that this means \eqref{jio} will still be valid), and then set
\begin{equation}\label{h1RV}
    H_1\coloneqq \sum_{n=1}^{T_1}\mathbbm{1}_{\mathcal{C}_n},
\end{equation}
where, writing $U_n^X$ for the horizontal coordinate of $X_n$ and $V^X_n$ for the vertical coordinate, so that $X_n=(U_n^X,V_n^X)$,
\[\mathcal{C}_n\coloneqq \left\{X_n=Y_n=Z_n,\: n<\theta_{hN},\: U^X_n\in (N/2,N],\: V^X_n\in [\varepsilon \log^{\alpha}(N/2),2\varepsilon \log^{\alpha}(N/2)]\right\}.\]
Thus $H_1$ is a random variable counting the number of triple collisions prior to time $\theta_{hN}$, subject to some geometric constraints. In particular, in the definition of $H_1$ we force the three random walks to meet on a tooth `far enough' from the origin $\0$ (this is the meaning of the condition $U^X_n\in (N/2,N]$) and, moreover, we ask the three walkers to collide at some `suitably-high site' on that tooth (this is the purpose of the condition $V^X_n\in [\varepsilon \log^{\alpha}(N/2),2\varepsilon \log^{\alpha}(N/2)]$). We highlight that, for establishing \eqref{mainLB} and our estimate on the collision rate when $\alpha<1$, it would suffice to consider the simpler random variable $H'_1\coloneqq \sum_{n=1}^{T_1}\mathbbm{1}_{\mathcal{E}_n}$, where
\begin{equation}\label{curlyedef}
\mathcal{E}_n\coloneqq\left\{X_n=Y_n=Z_n\in\mathbb{V}_N\backslash\mathbb{V}_{N/2},\:n< \theta_{hN}\right\},
\end{equation}
but we shall see later that for estimating the collision rate when $\alpha=1$, it is helpful to know that the three walkers have first met at one of the sites having the properties appearing in the definition of $H_1$. We claim that, taking $\x,\y,\z\in \mathbb{V}_N$, when $\alpha<1$, it holds that
\begin{equation}\label{localgoalsmallalpha}
    \mathbb{P}^{\x,\y,\z}(H_1\geq 1)\geq \frac{c}{\log(N)},
\end{equation}
and when $\alpha=1$, we have
\begin{equation}\label{localgoalalphaisone}
    \mathbb{P}^{\x,\y,\z}(H_1\geq 1)\geq \frac{c}{\log(N)\log\log(N)},
\end{equation}
for some constant $c>0$. (Note that, as per the discussion at the start of Section \ref{infinite}, these two lower bounds are already enough to establish that infinitely many collisions of the three random walks occur almost-surely.)

We will establish \eqref{localgoalsmallalpha} and \eqref{localgoalalphaisone} by means of the second moment method (see Lemma \ref{PZI}), which in particular allows us to write 
\begin{equation}\label{secondmomentfirstmeeting}
    \mathbb{P}^{\x,\y,\z}(H_1\geq 1)\geq\frac{\mathbb{E}^{\x,\y,\z}[H_1]^2}{\mathbb{E}^{\x,\y,\z}[H_1^2]}.
\end{equation}
In the next two lemmas, we bound the first and second moments of $H_1$, respectively.

\begin{lemma}\label{expH}
     There exists a constant $c>0$ such that, if $N\geq 2$ and $\x,\y,\z\in \mathbb{V}_N$ are such that $d(\x,\y)$, $d(\y,\z)$ and $d(\x,\z)$ are even, then 
    \[\mathbb{E}^{\x,\y,\z}[H_1]\geq \frac{c}{\log^{\alpha}(N)}.\]
\end{lemma}
\begin{proof}
    For a given $N$, denote by $S$ the set of all the potential sites where the three random walkers could meet in such a way as to contribute to $H_1$; that is, we set
    \[S\coloneqq \left\{\w=(w,\ell)\in\mathbb{V}_{hN}:\:w\in (N/2,N],\: \ell\in[\varepsilon \log^{\alpha}(N/2),2\varepsilon \log^{\alpha}(N/2)]\right\},\]
    and write
    \begin{equation*}
        \mathbb{E}^{\x,\y,\z}[H_1]=\sum_{n=1}^{T_1}\sum_{\w\in S}^{}p^{\mathbb{V}_{hN}}_n(\x,\w)p^{\mathbb{V}_{hN}}_n(\y,\w)p^{\mathbb{V}_{hN}}_n(\z,\w).
    \end{equation*}
    We will restrict the first sum to values of $n$ greater than $T_1'\coloneqq \lceil c_1 N^2\log^{\alpha}(N)\rceil$. (Note that, by our choice of $\varepsilon$, $c_1<(1-\varepsilon)c_2$.) In particular, applying \eqref{jio}, we can bound from below the above expression by 
    \begin{equation*}
        \sum_{n=T_1'}^{T_1}\sum_{\w\in S}^{}p^{\mathbb{V}_{hN}}_n(\x,\w)p^{\mathbb{V}_{hN}}_n(\y,\w)p^{\mathbb{V}_{hN}}_n(\z,\w)
        \geq \frac{c_3 (T_1-T_1')|S|}{N^3\log^{3\alpha}(N)}.
    \end{equation*}
Since $|S|\geq cN\log^\alpha(N)$ and $T_1-T_1'\geq c' N^2\log^\alpha (N)$, we arrive at 
    \[\mathbb{E}^{\x,\y,\z}[H_1]\geq \frac{c''N\log^\alpha(N)\times N^2\log^\alpha (N)}{N^3\log^{3\alpha}(N)}=\frac{c''}{\log^{\alpha}(N)},\]
    thereby completing the proof.    
\end{proof}

\begin{lemma}\label{secmomH}
   There exists a constant $c>0$ such that, if $N\geq 2$ and $\x,\y,\z\in \mathbb{V}_N$ are such that $d(\x,\y)$, $d(\y,\z)$ and $d(\x,\z)$ are even, then 
    \[\mathbb{E}^{\x,\y,\z}[H_1^2]\leq c\mathbb{E}^{\x,\y,\z}[H_1]\left(\log\log(N)+\log^{1-\alpha}(N)\right).\]
\end{lemma}
\begin{proof}
    Note that
    \begin{equation}\label{decH}
        \mathbb{E}^{\x,\y,\z}[H_1^2]=\mathbb{E}^{\x,\y,\z}[H_1]+2\sum_{n=1}^{T_1-1}\sum_{k=n+1}^{T_1}\mathbb{P}^{\x,\y,\z}(\mathcal{C}_n\cap \mathcal{C}_k).
    \end{equation}
    Moreover, using the Markov property and the definition of the events $\mathcal{C}_i$ we readily see that (for $n+1\leq k\leq T_1$)
    \begin{align*}
        \mathbb{P}^{\x,\y,\z}(\mathcal{C}_n\cap \mathcal{C}_k)&\leq \mathbb{P}^{\x,\y,\z}(\mathcal{C}_n)\sup_{\vv\in \mathbb{V}_{N}\backslash\mathbb{V}_{N/2}}\mathbb{P}^{\vv,\vv,\vv}(\mathcal{C}_{k-n})=\mathbb{P}^{\x,\y,\z}(\mathcal{C}_n)\sup_{\vv\in \mathbb{V}_N\backslash\mathbb{V}_{N/2}}\sum_{\w \in \mathbb{V}_{N}\backslash\mathbb{V}_{N/2}}^{}p^{\mathbb{V}_{hN}}_{k-n}(\vv,\w)^3.
    \end{align*}
    However, $p^{\mathbb{V}_{hN}}_{k-n}(\vv,\w)\leq p^{}_{k-n}(\vv,\w)$ and so, summing over $k$, we obtain
    \begin{align*}
    \lefteqn{\sum_{k=n+1}^{T}\mathbb{P}^{\x,\y,\z}(\mathcal{C}_n\cap \mathcal{C}_k)}\\
    &\leq \mathbb{P}^{\x,\y,\z}(\mathcal{C}_n)\sup_{\vv\in \mathbb{V}_{N}\backslash\mathbb{V}_{N/2}}\sum_{\w \in \mathbb{V}_{N}\backslash\mathbb{V}_{N/2}}^{}\sum_{k=n+1}^{T_1}p_{k-n}(\vv,\w)^3\\
        &\leq \mathbb{P}^{\x,\y,\z}(\mathcal{C}_n)\sup_{\vv\in \mathbb{V}_{N}\backslash\mathbb{V}_{N/2}}\sum_{\w \in \mathbb{V}_{N}\backslash\mathbb{V}_{N/2}}^{}\sum_{k=1}^{T_1}p_{k}(\vv,\w)^3\\
        &=\mathbb{P}^{\x,\y,\z}(\mathcal{C}_n)\sup_{\vv\in \mathbb{V}_{N}\backslash\mathbb{V}_{N/2}}\sum_{\w \in \mathbb{V}_{N}\backslash\mathbb{V}_{N/2}}^{}\sum_{k=1}^{K_0}p_{k}(\vv,\w)^3\\
        &\qquad\qquad+\mathbb{P}^{\x,\y,\z}(\mathcal{C}_n)\sup_{\vv\in \mathbb{V}_{N}\backslash\mathbb{V}_{N/2}}\sum_{\w \in \mathbb{V}_{N}\backslash\mathbb{V}_{N/2}}^{}\sum_{k=K_0+1}^{T_1}p_{k}(\vv,\w)^3,
    \end{align*}
where we set $K_0\coloneqq \lfloor 16\log^{3\alpha}(N)\rfloor$. Now, it follows from Lemma \ref{hku2} that   
    \begin{align*}
    \sup_{\vv\in \mathbb{V}_{N}\backslash\mathbb{V}_{N/2}}\sum_{\w \in \mathbb{V}_{N}\backslash\mathbb{V}_{N/2}}^{}\sum_{k=1}^{K_0}p_{k}(\vv,\w)^3&\leq \sup_{\vv\in \mathbb{V}_{N}\backslash\mathbb{V}_{N/2}}^{}\sum_{k=1}^{K_0}\frac{c}{k}\sum_{\w \in \mathbb{V}_{N}\backslash\mathbb{V}_{N/2}}p_{k}(\vv,\w)\\
    &\leq\sum_{k=1}^{K_0}\frac{c}{k}\leq c\log(K_0)\leq c\log\log(N),
    \end{align*}
and also 
\begin{align*}
\lefteqn{\sup_{\vv\in \mathbb{V}_{N}\backslash\mathbb{V}_{N/2}}\sum_{\w \in \mathbb{V}_{N}\backslash\mathbb{V}_{N/2}}^{}\sum_{k=K_0+1}^{T_1}p_{k}(\vv,\w)^3}\\
&\leq \sup_{\vv\in \mathbb{V}_{N}\backslash\mathbb{V}_{N/2}}\sum_{k=K_0+1}^{T_1}\frac{c}{k\log^{\alpha}(N)}\sum_{\w \in \mathbb{V}_{N}\backslash\mathbb{V}_{N/2}}^{}p_{k}(\vv,\w)\leq \sum_{k=K_0+1}^{T_1}\frac{c}{k\log^{\alpha}(N)}\leq c\log^{1-\alpha}(N).
\end{align*}
Therefore we obtain 
\[\sum_{k=n+1}^{T_1}\mathbb{P}^{\x,\y,\z}(\mathcal{C}_n\cap \mathcal{C}_k)\leq c\mathbb{P}^{\x,\y,\z}(\mathcal{C}_n)\left(\log\log(N)+\log^{1-\alpha}(N)\right),\]
and summing over $n$ finally yields 
\[\sum_{n=1}^{T_1-1}\sum_{k=n+1}^{T_1}\mathbb{P}^{\x,\y,\z}(\mathcal{C}_n\cap \mathcal{C}_k)\leq c\mathbb{E}^{\x,\y,\z}[H_1]\left(\log\log(N)+\log^{1-\alpha}(N)\right).\]
Combined with our earlier decomposition of the second moment of $H_1$ given at (\ref{decH}), this gives the desired result.
\end{proof}

Plugging back into (\ref{secondmomentfirstmeeting}) the estimates obtained in the preceding two lemmas, we arrive at
\begin{equation*}
    \mathbb{P}^{\x,\y,\z}(H_1\geq 1)\geq \frac{c\log^{-\alpha}(N)}{\log\log(N)+\log^{1-\alpha}(N)},
\end{equation*}
thus establishing (\ref{localgoalsmallalpha}) and (\ref{localgoalalphaisone}) simultaneously.

\subsubsection{Subsequent meetings}\label{subseqM}

Here we wish to estimate, again by means of the second moment method, the number of collisions of the three random walkers, when started at the \textit{same} location in $\mathbb{V}_{N}\backslash\mathbb{V}_{N/2}$ (at a site satisfying the constraint appearing in the definition of $H_1$), prior to time $\theta_{hN}$. Denote by $H_2$ the number of collisions of the three random walkers in $\mathbb{V}_{2N}\backslash\mathbb{V}_{N/2}$ prior to time $\theta_{hN}$, i.e.
\begin{equation}\label{h2RV}
    H_2\coloneqq \sum_{n=1}^{T_2} \mathbbm{1}_{\mathcal{E}'_n},
    \end{equation}
    where $T_2\coloneqq \lfloor \delta N^2\log^{\alpha}(N)\rfloor$ for some $\delta\in (0,1)$ and, similarly to \eqref{curlyedef},
    \begin{equation}\label{curlyedashdef}
   \mathcal{E}'_n\coloneqq \{X_n=Y_n=Z_n\in \mathbb{V}_{2N}\backslash\mathbb{V}_{N/2},\:n<\theta_{hN}\}.
\end{equation}
(Note that, since we assumed that $h\geq 2$, it holds that $\mathbb{V}_{2N}\subseteq \mathbb{V}_{hN}$.) Recalling from the proof of Lemma \ref{expH} that
\begin{equation}\label{Sdef}
S= \left\{\w=(w,\ell)\in\mathbb{V}_{hN}:\:w\in (N/2,N],\:\ell\in[\varepsilon \log^{\alpha}(N/2),2\varepsilon \log^{\alpha}(N/2)]\right\},
\end{equation}
we wish to show that there is a constant $c\in (0,1)$ such that, when $\alpha<1$, we have
\begin{equation}\label{B1}
    \mathbb{P}^{\x,\x,\x}(H_2\geq c \log^{1-\alpha}(N))\geq c \text{ for any }\x\in S,
\end{equation}
and for $\alpha=1$, we have
\begin{equation}\label{B2}
    \mathbb{P}^{\x,\x,\x}(H_2\geq c \log\log(N))\geq c \text{ for any }\x\in S.
\end{equation}
We achieve this by bounding from above the second moment of $H_2$ and from below its first moment. For the second moment, we can proceed exactly as in the proof of Lemma \ref{secmomH} (with $\mathbb{V}_N\backslash\mathbb{V}_{N/2}$ replaced by $\mathbb{V}_{2N}\backslash\mathbb{V}_{N/2}$) to deduce that, for all $\x\in\mathbb{V}_{N}$,
\begin{equation}\label{h22}
\mathbb{E}^{\x,\x,\x}[H_2^2]\leq c\mathbb{E}^{\x,\x,\x}[H_2]\left(\log\log(N)+\log^{1-\alpha}(N)\right).
\end{equation}
As for the lower bound on the first moment of $H_2$, we distinguish between the two regimes $\alpha<1$ and $\alpha=1$.

\paragraph{Case $\alpha<1$.} 
In this part of the argument, we will apply the lower bound of Lemma \ref{lowerboundheatkernel}. Specifically, we fix $C=1$, and then use that result to note that there are constants $c_1$, $c_2$ and $d'\geq 1$ for which: if $\x\in\mathbb{V}_{2N}$, $\y\in\mathbb{V}_{2N}\backslash\mathbb{V}_{N/2}$ and $n\in [2C+16\log^{3\alpha}(2N),4N^2\log^\alpha(2N)]$ are such that $d(\x,\y)-n$ is even and $n\leq c_1n^{1/2}\log^{-\alpha/2}(2N)$, then
\begin{equation}\label{bta}
  p_n^{\mathbb{V}_{2d'N}}(\x,\y)\geq \frac{c_2}{n^{1/2}\log^{\alpha/2}(2N)}.  
\end{equation}
We reselect $h$ to be the maximum of the previous $h$ and $2d'$; our earlier arguments will still be valid for this choice, and so is the above upper bound with $\mathbb{V}_{2d'N}$ replaced by $\mathbb{V}_{hN}$. With this preparation in place, we bound the first moment of $H_2$ as follows: for $\x\in\mathbb{V}_N\backslash\mathbb{V}_{N/2}$,
\begin{equation}\label{firstmomsubseqM}
    \mathbb{E}^{\x,\x,\x}[H_2]=\sum_{n=1}^{T_2}\sum_{\w\in \mathbb{V}_{2N}\backslash\mathbb{V}_{N/2}}p_n^{\mathbb{V}_{hN}}(\x,\w)^3
    \geq \sum_{n=2+16\log^{3\alpha}(2N)}^{T_2}\sum_{\w\in B(\x,c_1n^{1/2}\log^{-\alpha/2}(2N))\backslash\mathbb{V}_{N/2}}p_n^{\mathbb{V}_{hN}}(\x,\w)^3,
\end{equation}
where the inequality is justified by the observation that if $\w\in B(\x,c_1n^{1/2}\log^{-\alpha/2}(2N))$ and $n\leq T_2$, then
\[d(\x,\w)\leq c_1T^{1/2}_2\log^{-\alpha/2}(2N)\leq c_1\delta^{1/2}N\]
so that, by taking $\delta$ small enough, we can ensure $B(\x,c_1n^{1/2}\log^{-\alpha/2}(2N))\subseteq \mathbb{V}_{2N}$. Applying \eqref{bta}, we obtain
\begin{align*}
    \sum_{\w\in B(\x,c_1n^{1/2}\log^{-\alpha/2}(2N))\backslash\mathbb{V}_{N/2}}p_n^{\mathbb{V}_{hN}}(\x,\w)^3&\geq \frac{c}{n^{3/2}\log^{3\alpha/2}(2N)}|B(\x,c_1n^{1/2}\log^{-\alpha/2}(2N))\backslash\mathbb{V}_{N/2}|\\
    &\geq \frac{c}{n^{3/2}\log^{3\alpha/2}(N)}c_1n^{1/2}\log^{-\alpha/2}(N)\log^{\alpha}(N)\\
    &=\frac{c}{n\log^{\alpha}(N)},
\end{align*}
where for the volume bound of the second inequality, we can argue as in the proof of Lemma \ref{hku2}. Substituting this into (\ref{firstmomsubseqM}) yields
\[\mathbb{E}^{\x,\x,\x}[H_2] \geq \frac{c}{\log^{\alpha}(N)}\sum_{n=2+16\log^{3\alpha}(N)}^{T_2}n^{-1}\geq \frac{c}{\log^{\alpha}(N)}\log(T_2)\geq \bar{c}\log^{1-\alpha}(N).\]
Since we have from \eqref{h22} that
\[\mathbb{E}^{\x,\x,\x}[H^2_2]\leq c\mathbb{E}^{\x,\x,\x}[H_2]\log^{1-\alpha}(N),\]
using Lemma \ref{PZI}, we arrive at
\[\mathbb{P}^{\x,\x,\x}(H_2\geq (\bar{c}/2)\log^{1-\alpha}(N))\geq \mathbb{P}^{\x,\x,\x}(H_2\geq \mathbb{E}^{\x,\x,\x}[H_2]/2)\geq c,\]
thus establishing (\ref{B1}) upon taking the smaller between $\bar{c}/2$ and $c$ in the last display.

\paragraph{Case $\alpha=1$.}
The argument in this case is some sense simpler, but depends on the assumption that the first meeting happens at a location $\x\in S$. In particular, if $\x=(x,u)$, we suppose that $x\in (N/2,N]$ and $u\in [\varepsilon\log^\alpha(N/2),2\varepsilon\log^\alpha(N/2)]$. Similarly to the $\alpha<1$ case, we can bound 
\begin{equation*}
    \mathbb{E}^{\x,\x,\x}[H_2]=\sum_{n=1}^{T_2}\sum_{\w\in \mathbb{V}_{2N}\backslash\mathbb{V}_{N/2}}p_n^{\mathbb{V}_{hN}}(\x,\w)^3
    \geq \sum_{n=1}^{c_1 \log^{2\alpha}(N/2)}\sum_{\w\in B(\x,\delta n^{1/2})\backslash\mathbb{V}_{N/2}}p_n^{\mathbb{V}_{hN}}(\x,\w)^3,
\end{equation*}
where $c_1$ is the constant of Lemma \ref{HKonedim}, which is a lower heat kernel bound for a random walk on an interval, killed on hitting an endpoint, and $\delta\in (0,(\varepsilon c_1^{-1/2})\wedge c_2]$, with $c_2$ also being taken from Lemma \ref{HKonedim}.
In particular, it is clear that if $n\leq c_1\log^{2\alpha}(N/2)$, then $\delta n^{1/2}\leq \varepsilon  \log^\alpha(N/2)$, and so the ball $B(\x,\delta n^{1/2})$ is entirely contained within the tooth upon which $\x$ is located. Since killing on exiting an interval around $\x$ within the tooth it is located only serves to reduce the heat kernel, as compared to killing on exiting $\mathbb{V}_{hN}$, it follows from Lemma \ref{HKonedim} that
\begin{align*}
    \sum_{n=1}^{c_1 \log^{2\alpha}(N/2)}\sum_{\w\in B(\x,\delta n^{1/2})\backslash\mathbb{V}_{N/2}}p_n^{\mathbb{V}_{hN}}(\x,\w)^3&\geq \sum_{n=1}^{c_1 \log^{2\alpha}(N/2)}\sum_{\w\in B(\x,\delta n^{1/2})}\frac{c}{n^{3/2}}\\
    &\geq\sum_{n=1}^{c_1 \log^{2\alpha}(N/2)}\frac{c}{n}\\
    &\geq\bar{c}\log\log(N).
\end{align*}
Moreover, we know from \eqref{h22} that
\[\mathbb{E}^{\x,\x,\x}[H^2_2]\leq c\mathbb{E}^{\x,\x,\x}[H_2]\log\log(N),\]
and so, applying Lemma \ref{PZI}, we conclude in this case that
\[\mathbb{P}^{\x,\x,\x}(H_2\geq (\bar{c}/2)\log\log(N))\geq \mathbb{P}^{\x,\x,\x}(H_2\geq \mathbb{E}^{\x,\x,\x}[H_2]/2)\geq c,\]
thus establishing (\ref{B2}) upon taking the smaller between $\bar{c}/2$ and $c$ in the last display.

\subsubsection{Establishing Theorem \ref{thmgrowth}}\label{finalsec}

Towards proving Theorem \ref{thmgrowth}, we first consider the number of  collisions of $X$, $Y$ and $Z$ prior to time $\theta_{hN}$. In particular, we define
\begin{equation}\label{hRV}
    H(N)\coloneqq \sum_{n=1}^{\infty}\mathbbm{1}_{\mathcal{E}''_n},
\end{equation}
where, similarly to \eqref{curlyedef} and \eqref{curlyedashdef}, we set
\[\mathcal{E}''_n\coloneqq \{X_n=Y_n=Z_n,\:n<\theta_{hN}\}.\]
For this random variable, we have the following lemma.

\begin{lemma}\label{hittimeversion} For $\alpha\in (0,1]$, it almost-surely holds that
\[\limsup_{n\rightarrow\infty}\frac{H(N)}{\log^{1-\alpha}(N)\vee \log\log(N)}>0,\]
\end{lemma}
\begin{proof}
We start by considering the case $\alpha<1$. For $\x,\y,\z\in \mathbb{V}_N$ (with $d(\x,\y)$, $d(\y,\z)$ and $d(\x,\z)$ even) and $\eta$ a constant to be determined later, we bound
\begin{align*}
    \mathbb{P}^{\x,\y,\z}(H(N)\geq \eta\log^{1-\alpha}(N))&\geq \mathbb{P}^{\x,\y,\z}(H_1\geq 1,\: H(N)\geq \eta\log^{1-\alpha}(N))\\
    &=\mathbb{E}^{\x,\y,\z}\left[\mathbbm{1}_{\{\sigma\leq T_1\wedge \theta_{hN},\:X_\sigma\in S\}}\mathbb{P}^{\x,\y,\z}\left(\sum_{n=\sigma}^{\infty}\mathbbm{1}_{\mathcal{E}''_n}\geq \eta\log^{1-\alpha}(N)\:\vline\: \mathcal{F}_{\sigma}\right)\right],
\end{align*}
where $T_1$, $H_1$ and $S$ were defined at \eqref{t1def}, \eqref{h1RV} and \eqref{Sdef}, respectively, and $\sigma$ is the first time at which the three random walkers collide (so that, in particular, $H_1\geq 1$ if, and only if, $\sigma\leq T_1\wedge \theta_{hN}$ and $X_\sigma\in S$). Now, on the event $\{\sigma\leq T_1\wedge \theta_{hN},\:X_\sigma\in S\}$, we have 
\begin{align*}
\mathbb{P}^{\x,\y,\z}\left(\sum_{n=\sigma}^{\infty}\mathbbm{1}_{\mathcal{E}''_n}\geq \eta\log^{1-\alpha}(N)\:\vline\: \mathcal{F}_{\sigma}\right)&\geq \mathbb{P}^{\x,\y,\z}\left(\sum_{n=\sigma}^{\sigma+T_2}\mathbbm{1}_{\mathcal{E}''_n}\geq \eta\log^{1-\alpha}(N)\:\vline\: \mathcal{F}_{\sigma}\right)\\
&\geq  \inf_{\vv\in S}\mathbb{P}^{\vv,\vv,\vv}\left(H_2\geq \eta\log^{1-\alpha}(N)\right),
\end{align*}
where $H_2$ was defined in (\ref{h2RV}), and, by (\ref{B1}), provided we take $\eta$ small enough, the right-hand side here is bounded below by a positive constant $c$. Moreover, it follows from (\ref{localgoalsmallalpha}) that
\[\mathbb{P}^{\x,\y,\z}(\sigma\leq T_1\wedge \theta_{hN},\:X_\sigma\in S)=\mathbb{P}^{\x,\y,\z}(H_1\geq 1)\geq \frac{c}{\log(N)}.\]
Hence we arrive at
\begin{equation*}
\mathbb{P}^{\x,\y,\z}(H(N)\geq \eta\log^{1-\alpha}(N))\geq \frac{c}{\log(N)}.
\end{equation*}
As for the case $\alpha=1$, arguing as above yields
\begin{align*}
\mathbb{P}^{\x,\y,\z}(H(N)\geq \eta \log\log(N))&\geq \mathbb{P}^{\x,\y,\z}(H_1\geq 1)\inf_{\vv\in S}\mathbb{P}^{\vv,\vv,\vv}\left(H_2\geq \eta \log\log(N)\right)\nonumber\\
&\geq \frac{c}{\log(N)\log\log(N)},
\end{align*}
where for the second inequality we have applied \eqref{localgoalalphaisone} and \eqref{B2}.

Completing the proof is now a straightforward application of the conditional Borel-Cantelli lemma, similar to that used to deduce \eqref{bcapp}. 
\end{proof}
Finally, we give the proof of Theorem \ref{thmgrowth}.
\begin{proof}[Proof of Theorem \ref{thmgrowth}]
Given Lemma \ref{hittimeversion}, it will suffice to compare 
$\theta_{hN}$ with a deterministic time. In particular, from the commute time identity (see \cite[Proposition 10.6]{LPW}, for example), we have that
\[\mathbb{P}^{\mathbf{0},\mathbf{0},\mathbf{0}}\left(\theta_{hN}>N^4\right)\leq \mathbb{P}^{\mathbf{0}}\left(\theta_{hN}^X>N^4\right)\leq\frac{3R(\mathbf{0},\mathbb{V}_{hN}^c)|\mathbb{V}_{hN}|}{N^4}\leq \frac{cN^2\log^\alpha(N)}{N^4}.\]
Since the upper bound here is summable, the standard Borel-Cantelli lemma gives that $\theta_{hN}\leq N^4$ eventually, almost-surely. Hence, recalling the definition of the random variable $H(N)$ given at (\ref{hRV}), we see that, eventually, almost-surely
\[C_N=\sum_{n=1}^{N}\mathbbm{1}_{\{X_n=Y_n=Z_n\}}\geq \sum_{n=1}^{\infty}\mathbbm{1}_{\{X_n=Y_n=Z_n, n<\theta_{hN^{1/4}}\}}=H(N^{1/4}),\]
and so we obtain the result from Lemma \ref{hittimeversion}.
\end{proof}

\appendix 

\section{Appendix}

Here we collect some elementary results that were applied in the article. We start by recalling a general version of the Borel-Cantelli lemma, which we used to establish the infinite triple collision property when $\alpha\leq 1$.

\begin{lemma}[{\cite[Corollary 7.20]{Kall}}]\label{BCL}
 Let $(\mathcal{F}_n)_{n\in\mathbb{N}_0}$  be a filtration and  $(A_n)_{n\in\mathbb{N}}$ a sequence of events with $A_n\in \mathcal{F}_n$. Then, almost-surely,
\[\left\{A_n \text{ infinitely often}\right\}=\left\{\sum_{n\geq 1}\mathbb{P}(A_n|\mathcal{F}_{n-1})=\infty\right\}.\]
\end{lemma}

Next, we recall the Paley–Zygmund inequality, a proof of which is included for completeness.

\begin{lemma}[{\cite[Lemma 4.1]{Kall}}]\label{PZI}
    Let $X$ be a non-negative random variable with finite mean and $\eta\in (0,1)$. Then
    \[\mathbb{P}\left(X\geq\eta \mathbb{E}[X]\right)\geq (1-\eta)^2\frac{\mathbb{E}[X]^2}{\mathbb{E}[X^2]}.\]
\end{lemma}
\begin{proof}
    Using the Cauchy–Schwarz inequality, we can bound
    \begin{equation*}
        \mathbb{E}[X]\leq \eta \mathbb{E}[X]+\mathbb{E}[X\mathbbm{1}_{\{X\geq \eta\mathbb{E}[X]\}}]\leq \eta \mathbb{E}[X] + \mathbb{E}[X^2]^{1/2}\mathbb{P}(X\geq \eta \mathbb{E}[X])^{1/2},
    \end{equation*}
    whence $\mathbb{E}[X](1-\eta)\leq \mathbb{E}[X^2]^{1/2}\mathbb{P}(X\geq\eta \mathbb{E}[X])^{1/2}$; now divide by $\mathbb{E}[X^2]^{1/2}$ and take squares.
\end{proof}

In estimating the collision growth rate in the case $\alpha=1$ in Section \ref{subseqM}, we applied an estimate on the transition density of a killed random walk on an interval. We complete the appendix by stating and proving this. Fix an integer $L\geq 2$. Let $X^L$ be a nearest-neighbour random walk on $\{0,\dots, L\}$, killed on hitting $\{0,L\}$, and $(q^L_n(x,y))_{x,y\in\{0,1,\dots,L\},\:n\geq 0}$ be its transition density, defined analogously to \eqref{killedtddef}. We then have the following result. (Since the proof is similar to, but easier than, that of Lemma \ref{lowerboundheatkernel}, we will be brief with the details.)

\begin{lemma}\label{HKonedim}
    Fix $\varepsilon \in (0,1/2)$. There exist constants $c_1,c_2,c_3\in(0,\infty)$ such that: for all $L\geq 2$, $1\leq n\leq c_1 L^2$, $x\in \{\varepsilon L,\dots, (1-\varepsilon)L\}$, and $y$ satisfying $|x-y|\leq c_2\sqrt{n}$ and $|x-y|-n$ is even,
    \[q^L_n(x,y)\geq \frac{c_3}{\sqrt{n}}.\]
\end{lemma}
\begin{proof}
First, let $m$ be a positive integer and suppose $y\in \{m+1,m+2,\dots,L-(m+1)\}$. It is then the case that, up to the time it exits $B(y,m)$, $\tau^L_{B(y,m)}$ say, $X^L$ behaves the same as a simple random walk on $\mathbb{Z}$. (Here, we write $B(y,m)$ for the ball of radius $m$ about $y$, defined similarly to \eqref{vdef}.) Thus we readily compute that
\[\mathbb{E}_z^L\tau^L_{B(y,m)}\leq \mathbb{E}_y^L\tau^L_{B(y,m)}=(m+1)^2,\qquad \forall z\in B(y,m),\]
where $\mathbb{E}_z^L$ gives the expectation with respect to the law of $X^L$ when this process is started from $z$, and the equality is a straightforward application of the optimal stopping theorem. Following the argument leading to \eqref{exitprob} and \eqref{odlower}, we deduce from this that
\[\mathbb{P}^L_y\left(\tau^L_{B(y,m)}\geq m^2/2\right)\geq \frac{1}{2},\]
where $\mathbb{P}^L_y$ is the law of $X^L$ started from $y$, and also
\[q^{L,B(y,m)}_{m^2}(y,y)\geq \frac{c}{m},\]
where $q^{L,B(y,m)}$ is the transition density of $X^L$ killed on exiting $B(y,m)$.

Now, let $x$, $y$ and $n$ be as in the statement of the lemma, and set $m=\sqrt{2n}$. We then have that 
\[B(y,m)\subseteq B(x,m+c_2\sqrt{n})\subseteq B(x,2m)\]
for $c_2$ chosen suitably small. Moreover, since $m=\sqrt{2n}\leq \sqrt{2c_1}L$ and $x\in \{\varepsilon L,\dots, (1-\varepsilon)L\}$, we can ensure that $B(x,2m)\subseteq \{1,\dots,L-1\}$ by choosing $c_1$ small. Hence, similarly to \eqref{ilb}, we have that
\[q^L_n(x,y)\geq q^{L,B(x,2m)}_n(x,y)\geq \mathbb{P}^L_x\left(\tau^L_{\{y\}^c}\leq n,\:\tau^{L}_{B(x,2m)}>\tau^L_{\{y\}^c}\right)q^{L,B(y,m)}_{2n}(y,y).\]
By our choice of $m$, it holds that
\[q^{B(y,m)}_{2n}(y,y)=q^{B(y,m)}_{m^2}(y,y),\]
and we have from the conclusion of the first paragraph that this is bounded below by $c/m=c/(2\sqrt{n})$. Moreover, applying the bounds used in the relevant part of the proof of Lemma \ref{lowerboundheatkernel}, we have that
\begin{align*}
    \mathbb{P}^L_x\left(\tau^L_{\{y\}^c}\leq n,\:\tau^{L}_{B(x,2m)}>\tau^L_{\{y\}^c}\right)&=1-\mathbb{P}^L_x\left(\tau^L_{\{y\}^c}>n,\: \tau^L_{B(x,2m)}> \tau^L_{\{y\}^c}\right)-\mathbb{P}^L_x\left(\tau^L_{\{y\}^c}>\tau^{L}_{B(x,2m)}\right)\\ 
    &\geq 1-\mathbb{P}^L_x\left(\tau^L_{\{y\}^c\cup B(x,2m)}\geq n\right)-\mathbb{P}^L_x\left(\tau^L_{\{y\}^c}>\tau^{L}_{B(x,2m)}\right)\\
    &\geq 1-\frac{2c_2\sqrt{n}(4m+1)}{n}-\frac{c_2\sqrt{n}}{m+1},
\end{align*}
which can be made greater than $1/2$ by taking $c_2$ sufficiently small. The result follows.
\end{proof}

\bibliographystyle{amsplain}
\bibliography{collision}

\end{document}